\documentclass[review]{siamart1116}
\usepackage{amsfonts}
\usepackage{amsmath}
\usepackage{algorithm,algorithmic}
\usepackage{bbm}
\usepackage{subfigure}
\usepackage{color}

\usepackage{lipsum}
\usepackage{amsfonts}
\usepackage{graphicx}
\usepackage{epstopdf}
\usepackage{algorithmic}
\ifpdf
  \DeclareGraphicsExtensions{.eps,.pdf,.png,.jpg}
\else
  \DeclareGraphicsExtensions{.eps}
\fi

\numberwithin{theorem}{section}

\newcommand{\TheTitle}{Bayesian Filtering for ODEs with bounded derivatives} 
\newcommand{\TheAuthors}{E. Magnani, H. Kersting, M. Schober and P. Hennig}

\headers{\TheTitle}{\TheAuthors}

\title{{\TheTitle}}

\author{
  Emilia Magnani
  \and
  Hans Kersting
  \and
  Michael Schober
  \and
  Philipp Hennig \thanks{Max-Planck-Institute for Intelligent Systems, T\"ubingen, Germany
    (\email{emagnani@tue.mpg.de}, \email{hkersting@tue.mpg.de}, \email{mschober@tue.mpg.de}, \email{ph@tue.mpg.de}).}.
}

\usepackage{amsopn}

\ifpdf
\hypersetup{
  pdftitle={\TheTitle},
  pdfauthor={\TheAuthors}
}
\fi

\begin{document}

\maketitle

\begin{abstract}
Recently there has been increasing interest in probabilistic solvers for ordinary differential equations (ODEs) that return full probability measures, instead of point estimates, over the solution and can incorporate uncertainty over the ODE at hand, e.g.~if the vector field or the initial value is only approximately known or evaluable.
The \emph{ODE filter} proposed in \cite{Kersting2016UAI,2016arXiv161005261S} models the solution of the ODE by a Gauss-Markov process which serves as a prior in the sense of Bayesian statistics.
While previous work employed a Wiener process prior on the (possibly multiple times) differentiated solution of the ODE and established equivalence of the corresponding solver with classical numerical methods, this paper raises the question whether other priors also yield practically useful solvers.
To this end, we discuss a range of possible priors which enable fast filtering and propose a new prior--the Integrated Ornstein Uhlenbeck Process (IOUP)--that complements the existing Integrated Wiener process (IWP) filter by encoding the property that a derivative in time of the solution is bounded in the sense that it tends to drift back to zero.
We provide experiments comparing IWP and IOUP filters which support the belief that IWP approximates better divergent ODE's solutions whereas IOUP is a better prior for trajectories with bounded derivatives.
\end{abstract}



\section{Introduction}
Ordinary differential equations (ODEs) are an extensively
studied mathematical problem.  Dynamical systems are described by  ODEs, and the application of these extends to many different fields of science and engineering.  
Given the initial value proplem (IVP) on the time interval $[0,T]\subset \mathbb{R}$ \begin{equation} \label{IVP}
 x'(t) = f(t,x(t)), \qquad x(0) = x_0 \in \mathbb{R}^d  \end{equation}
with $f: [0,T]\times\mathbb{R}^d\longrightarrow\mathbb{R}^d$ globally Lipschitz continuous, we want to compute a numerical approximation $\hat{x}: [0,T]\longrightarrow\mathbb{R}^d$ of the solution $x:[0,T]\longrightarrow\mathbb{R}^d$ using a probabilistic approach. Solving ODEs is one of the major tasks of numerical mathematics.
Classical numerical ODE solvers (e.g. Runge Kutta methods) construct the solution $\hat{x}$ by iteratively evaluating the vector field $f$  at discrete times $t_1,\dots,t_N$ at numerical approximations $\hat{x}(t_n)$ and collecting information on the derivative $x'$.
As pointed out in the past \cite{poincare1896,diaconis88:_bayes,ohagan92:_some_bayes_numer_analy}, to incorporate uncertainties that arise due to the inaccurate observation of $x'(t)=f(t,\hat{x}(t))$ and consequently accumulated numerical error,  we can  interpret the observed evaluations of $f$ as (potentially noisy) data for statistical inference. 
For quantifying the uncertainty on $\hat{x}$ we model the analytic solution as a Gauss-Markov stochastic process $X$ defined on a probability space $(\Omega,\mathcal{F},\mathbb{P})$. The iterative algorithm then outputs the law of $X$, i.e. the probability measure that the process induces on the collection of functions from $[0,T]$ to the state space $\mathbb{R}^d$. This probability measure contains the belief over $x$.\\
From this viewpoint, unknown numerical quantities are modelled as random variables and numerical algorithms return probability distributions instead of points estimates. This is a probabilistic approach which carries out numerical computations in a statistically interpretable way by combining probability theory and numerics, and permits to quantify the uncertainties for the computations. This idea became a dynamic research area which includes optimization, linear algebra and differential equations and is called \textit{Probabilstic Numerics} (PN) \cite{HenOsbGirRSPA2015}.\\
A PN solver for ODEs was introduced by \cite{HennigAISTATS2014} who recast the inference of the solution as a Gaussian process (GP) regression.
However, their original solver did not satisfy the polynomial convergence rates of classical solvers. 
In parallel development, \cite{o.13:_bayes_uncer_quant_differ_equat} proposed a similar solver of similar structure, which uses a Monte Carlo, instead of a closed-form Gaussian, updating scheme.
This solver produced solutions of linear order, but still not of polynomial order.
In \cite{schober2014nips}, the polynomial convergence rate was established by choosing an Integrated Wiener Process (IWP) prior.
This solver was further amended by adding calibrated uncertainty quantification by Bayesian quadrature \cite{Kersting2016UAI} and using the Markov property of the IWP to recast it as a filter, thereby drastically speeding up the probabilistic computations \cite{Kersting2016UAI,2016arXiv161005261S}. 
While the resulting filter is the fastest probabilistic ODE solver available, it outputs solely Gaussian probability measures.
Another novel method introduced in \cite{CGSSZ15} provides a more flexible non-parametric uncertainty quantification by defining a generative process to sample random numerical solutions of the ODE, obtained by adding Gaussian noise to the solution of a classical solver.
The computational cost of this method per sample, however, is equal to the cost of the underlying solver, which makes it impractical for cost-sensitive applications.\\
In this paper we recall the solver in \cite{schober2014nips,Kersting2016UAI} which employs a $q$-times integrated Wiener process prior on the solution and introduce instead an Integrated Ornstein Uhlenbeck (IOU) process prior.  This new model improves the algorithm for ODEs whose trajectories have a higher derivative which is expected to gravitate back to zero. In the following, we will call such systems 'ODEs with bounded derivatives'. 

\section{Gaussian filtering for ODEs} 
\subsection{Bayesian filtering problem}
Assume we have a \textit{signal process} $X = \\ \lbrace{X_t : 0\leq t\leq T}\rbrace$ which is not directly observable, solution of the stochastic differential equation (SDE) \begin{equation}
dX_t = b(t, X_t)dt + \sigma(t,X_t)dW_t
\end{equation} where $W$ denotes a Wiener process. Instead, we can only see an \textit{observation process} $Z = \lbrace{Z_t : 0\leq t\leq T}\rbrace$ that is a transformation of $X$.
 Both processes are defined on some probability space $(\Omega,\mathcal{F},P)$. 
The goal of the filtering problem is to compute, given the observations
 $ \lbrace{Z_s : 0\leq s\leq t}\rbrace$, the $L^2$-optimal estimate $\hat{X}_t$ of $X_t$ for $t \in [0,T]$. Accordingly, $\hat{X}_t$ has to be measurable w.r.t. the natural filtration $\mathcal{F}^Z_t$ and minimize the mean-square distance between $X_t$ and the candidates in  $L^2(\Omega,\mathcal{F}^Z_t,P)$ :  \begin{equation}
 \mathbb{E} {\left \Vert X_t - \hat{X}_t \right \Vert}^2 = \min_{Y \in L^2(\Omega,
 \mathcal{F}^Z_t,P)} \mathbb{E} {\left \Vert X_t - Y \right \Vert }^2.
 \end{equation}
 The best mean square estimate of $X_t$ given $Z_{s \leq t }$ is known to be  the conditional expectation $\hat{X}_t = \mathbb{E}[X_t \mid \mathcal{F}^Z_t]$ \cite{karatzas1991brownian,oksendal2003stochastic}. 
 In a Bayesian spirit, we seek to compute the posterior distribution of $X_t$ given the current and previous noisy observations $Z_{s \leq t}$. This can be done by using Bayes' rule \begin{equation}
 p(X \mid Z) = \frac{p(Z \mid X) p(X)}{p(Z)}.
 \end{equation} The distribution $p(Z \mid X) $ is the measurement model which describes the noisy relationship between the true signal and the observations, and $p(X)$ is the prior containing preliminary information on $X$ before any observation are taken into account. A natural question is then: which prior should we put on the signal process $X$?
Due to the computationally advantageous properties of Gaussian distributions which allow for fast inference by matrix computations we assume $X$ to be a Gaussian process, i.e. to have finite joint Gaussian distribution. Moreover we want to iteratively update the solution from $X_{t_i}$ to $X_{t_{i+1}}$ without taking into account the previous observations at time steps $t_j < t_i$. This is possible when $X$ is also a Markov process, since otherwise the computational cost per step would exponentially increase in the number of previous steps.
Sample continuous Gauss-Markov processes can be written as the solution of the linear time-invariant stochastic differential equation (LTI-SDE) of the form
\begin{equation}
dX_t = FX_tdt + LdW_t \label{eq:SDE}
\end{equation}
with Gaussian initial condition $X_0 \sim \mathcal{N}(m_0,P_0)$.
 The matrix $F \in \mathbb{R}^{(q+1)\times(q+1)}$ is the so called drift matrix and $L\in \mathbb{R}^{q+1}$ is the diffusion matrix of the process.
 The solution $X$ of \cref{eq:SDE} depends on the form of $F$ and $L$. 
With this conditions the distribution of the estimate $\hat{X}_t$ can be computed by Kalman filtering \cite{Sarkka2013bayesian}, an iterative algorithmic implementation for the linear filtering problem with Gaussian model assumptions. Since it is essentially an application of Bayes' rule, it is determined by the choice of the prior $p(X)$ and the measurement model $p(Z \mid X)$. Previous work \cite{Kersting2016UAI,schober2014nips,2016arXiv161005261S} employed an IWP prior, discussed in \cref{sec:IWP}, whereas in this paper we will analyse Bayesian filtering for ODEs with an IOUP prior.   Due to the properties of Gaussian distributions, together with Gaussian measurement model assumptions, we can analytically compute the matrices for Gaussian filtering  for LTI SDE like \cref{eq:SDE} (see \cite{sarkka2006thesis}).

\subsection{Filtering for ODEs}\label{filteringODEs}
We use Bayesian filtering for computing a numerical estimation $\hat{x}$ of the solution of the IVP \cref{IVP}. We treat the imprecise measurements $z_n = f(t_n,\hat{x}(t_n))$ as noisy observations. Since $\hat{x}_t$ is a numerical approximation of the true solution $x_t$ at time $t$ the evaluations of $f$ are indeed noisy. Therefore, we model the solution and the first $q$ derivatives  $(x, x', x'', \dots ,x^{(q)}) : [0,T]\rightarrow \mathbb{R}^{(q+1)\times d}$  as a draw of a Gauss Markov stochastic process  ${X = (X_t)}_{t \in [0,T]} = {({X_t}^{(0)}, \dots , {X_t}^{(q)})}_{t \in [0,T]}$. 
Hence, the SDE in \cref{eq:SDE}, which describes the dynamics of $X_t$,  has the following form 

\begin{equation}\label{eq:matrices}
\begin{pmatrix}
dX_t^{(0)}\\dX_t^{(1)}\\ \vdots\\dX_t^{(q-1)}\\dX_t^{(q)}\\
\end{pmatrix} = \overbrace{\begin{pmatrix}
0&1&0&\cdots&0\\
0&0&1&\cdots&0\\
\vdots& & \ddots&\ddots&\vdots \\
0& &\cdots&0&1\\
f_{q,0}&f_{q,1}&\dots&\dots&f_{q,q}\\
\end{pmatrix}}^F
\begin{pmatrix}
X_t^{(0)}\\X_t^{(1)}\\ \vdots\\X_t^{(q-1)}\\X_t^{(q)}\\
\end{pmatrix}dt
+  \overbrace{\begin{pmatrix}
0\\0\\\vdots\\0\\\sigma\\
\end{pmatrix}}^L dW_t.
\end{equation} 
 
The drift matrix $F \in \mathbb{R}^{(q+1)\times(q+1)}$ has non-zero entries only on the first-off diagonal and on the last row. This form of $F$ guarantees that the $i$-th component is the derivative of the $(i-1)$-component in the state space.  The parameters $\lbrace f_{q,0}, \dots, f_{q,q} \rbrace$ determine the process employed on the $q$-the component of the state space $X$: they are all equal to zero for an integrated Wiener process, equal to zero apart from $f_{q,q} =  \theta$ for the integrated OUP with some $\theta <0$, whereas for the Matern process \cite{hartikainen2010kalman} we have $f_{q,i} \neq 0 , i=0,\dots q$ . 
Inference in \cref{eq:matrices} with linear Gaussian measurements models $z_n = x'_n +r, r\sim\mathcal{N}(0,R)$ gives rise to the Kalman Filter equations implemented in Algorithm \ref{alg:algorithm1}. The algorithm outputs the \textit{filtering distribution} $p(X_{t_n} | z_{[n]})$, where $z_{[n]}:=\lbrace z_i |i\leq n \rbrace$.

\begin{algorithm} 
\caption{Kalman ODE filtering}
\label{alg:buildtree}
\begin{algorithmic}[1]
\STATE Choose $q \in \mathbb{N}$, $F$ and $L$ \COMMENT{Choose prior model}\label{line1}
\STATE{Initialize $t_0=0$, $m_0=\mathbb{E}[x_0]$, $P_0=Cov[x_0]$ }
\STATE{Choose $h:=t_n-t_{n-1}$}
\STATE{Compute $A(h)$ and $Q(h)$}
\FOR{$n=1$ to $N$}
\STATE{$m^{-}_{t_n} \leftarrow  A(h)m_{t_{n-1}}$} \COMMENT{Predict}\label{line6}
\STATE{$ P^{-}_{t_n} \leftarrow A(h)P_{t_{n-1}}A(h)^T + Q(h) $ }\label{line7}
\STATE{$ z_n \leftarrow f(t_n, m^{-}_{t_n}) $} \COMMENT{Collect evaluations ('data')} 
\STATE{$ v_n \leftarrow  z_n - Hm^{-}_{t_n} $}\label{line9} \COMMENT{Update}
\STATE{$ S_n \leftarrow  HP^{-}_{t_n}H^T + R $}
\STATE{$ K_n \leftarrow P^{-}_{t_n}H^TS^{-1} $}
\STATE{$ m_{t_n} \leftarrow m^{-}_{t_n} + K_{n}v_n $}
\STATE{$ P_{t_n} \leftarrow P^{-}_{t_n} - K_nS_nK_n^T $}\label{line13}
\ENDFOR
\RETURN ${(m_{t_n},P_{t_n})}_{n=0,\dots N}$
\end{algorithmic} \label{alg:algorithm1}
\end{algorithm}
The transition matrices $A(h) \in \mathbb{R}^{(q+1)\times(q+1)}$ and $Q(h) \in \mathbb{R}^{(q+1)\times(q+1)}$ in lines \ref{line6}-\ref{line7} are given by (\cite{sarkka2006thesis}) \begin{equation} \label{eq:A}
  A(h) = \exp(Fh)
  \end{equation}
  \begin{equation} \label{eq:Q}
  Q(h) = \int_0^t \exp(F\tau)LL^T\exp(F\tau)^T d\tau
  \end{equation}
  where $F$ and $L$ come from \cref{eq:matrices}. The matrix $H$ is the measurement model matrix (in probabilistic terms $p(z_k \mid x_k)\sim\mathcal{N}(H_kx_k,R)$), in our case the $n$-th unit vector $e^T_n \in \mathbb{R}^{1\times q}$.
In line \ref{line1} a prior model is selected. Lines \ref{line6}-\ref{line7}  compute the \textit{prediction distribution} $p(X_{t_n}\mid z_{[n-1]})=\mathcal{N}(m^{-}_{t_n}, P^{-}_{t_n})$, which depends from the form of the matrices $A(h)$ and $Q(h)$ and is therefore strictly related to the choice of the prior. Lines \ref{line9}-\ref{line13} are known as \textit{update step} and compute the updated distribution $p(X_{t_n}\mid z_{[n]})=\mathcal{N}(m_{t_n}, P_{t_n})$. In the experiments we will assume for simplicity $R=0$, i.e. the numerical approximation is exact.

\subsection{Prior selection: \textit{Integrated Wiener process}} \label{sec:IWP}
The standard Brownian motion $W_t$ is a Gauss Markov process with mean $\mu = 0$ and variance  $Var(W_t) = \sigma^2t$. Samples of the process are depicted in \cref{fig:samples} (b).
Previous work \cite{Kersting2016UAI,schober2014nips,2016arXiv161005261S} worked with an IWP prior for probabilistic ODE solvers. In this case $A(h)$ and $Q(h)$ are given by \cite{schober2014nips} \begin{equation}
A(h)_{i,j} = e^{Ah}_{i,j} = \mathbbm{1}_{j\geq i}\frac{h^{j-i}}{(j-i)!}
\end{equation}
\begin{equation}
Q(h)_{i,j} = \sigma^2 \frac{h^{2q+1-i-j}}{(2q+1-i-j)(q-i)!(q-j)!}.
\end{equation}
This work showed the useful uncertainty quantification for the solution $x$ with only a small computational overhead compared to classical numerical methods. In particular, Schober et al.\cite{schober2014nips} constructed a class of probabilistic solvers whose posterior mean have the same analytic structure of the one of the classic numerical Runge-Kutta methods (for first, second and third order), thus acquiring their good properties.

\section{Ornstein Uhlenbeck process}
Can we use other Gauss Markov priors over the solution which preserve good convergence rates? In this paper we model the solution function and the first $q$  derivatives as a draw from an \textit{integrated Ornstein Uhlenbeck process (IOUP)} to perform inference in ODEs. We show that this model is competitive with the IWP solver discussed in the previous section.
The Ornstein Uhlenbeck process (OUP) is a Gauss-Markov process whose dynamic is governed by the following SDE 
\begin{equation}
dX_t = \theta(X_t-\mu)dt + \sigma dW_t
\end{equation}
where $\theta <0$ and $W_t$ denotes a standard Wiener process.
The OUP converges to its stationary distribution $\mathcal N(\mu,
{-\sigma^2}/{2\theta})$ in the sense that $p(X_t) \to \mathcal N(\mu,{-\sigma^2}/{2\theta})$, for $t \to \infty$. 
The difference between the Wiener process and the Ornstein Uhlenbeck is the drift term $\theta(X_t-\mu)$, which is constant for the Wiener process and dependent on the parameter $\theta$ for OUP.
If the current value of the process is smaller than the value of the mean $\mu$, the drift $\theta(X_t-\mu)$ is positive and the process tends to drift back to $\mu$. The same (symmetric) effect occurs if the current value of $X_t$ is greater than $\mu$. In other words, the samples evolve around the mean which acts as an equilibrium level. Some samples of the process are shown in \cref{fig:samples} (a).
This property is called "mean reverting" and the magnitude of $\theta$ determines the strength of this drift. For small values of $\mid\theta\mid$ the effect vanishes, whereas, for large $\mid\theta\mid$, $X_t$ stays close to the mean $\mu$ with high probability. This behaviour justifies the fact that the variance of the OUP is bounded, in fact $Var(X_t) = \sigma^2/2\theta(e^{2\theta t}-1)$ bounded by $\mid{\sigma^2}/{2\theta}\mid$, whereas the Wiener process variance is unbounded ($Var(W_t)=\sigma^2t$) and the samples don't tend to remain  around the mean.  \\

\begin{figure} \label{fig:samples}
 \centering
 \subfigure[Samples of an Ornstein Uhlenbeck process with initial position $\mu=0$, variance $\sigma^2=1$ and $\theta=-1$.The samples tend to remain close to the mean $\mu$ (mean reverting property).] 
   {\includegraphics[width=6.3cm]{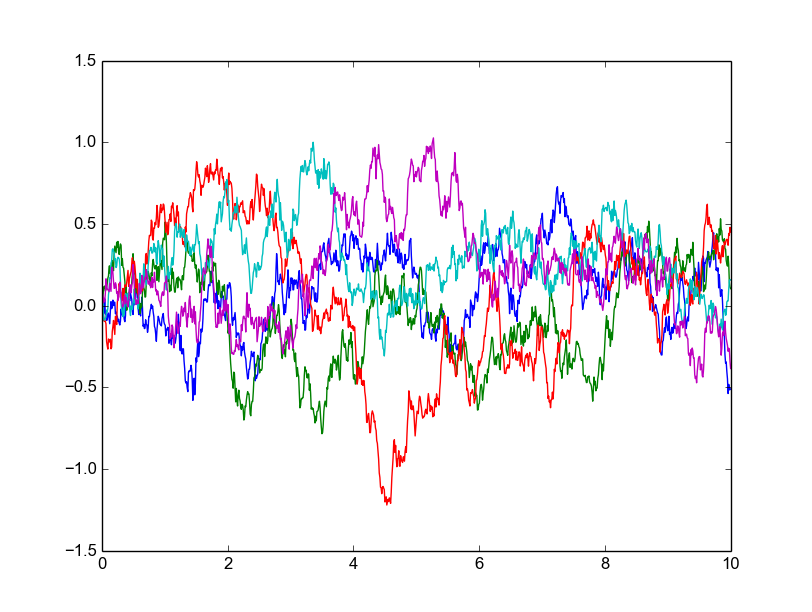}}
 \hspace{0.3mm}
 \subfigure[Samples of a Wiener process with initial mean $\mu=0$ and variance $\sigma^2=1$. The samples do not 'drift back' to the mean.]
   {\includegraphics[width=6.3cm]{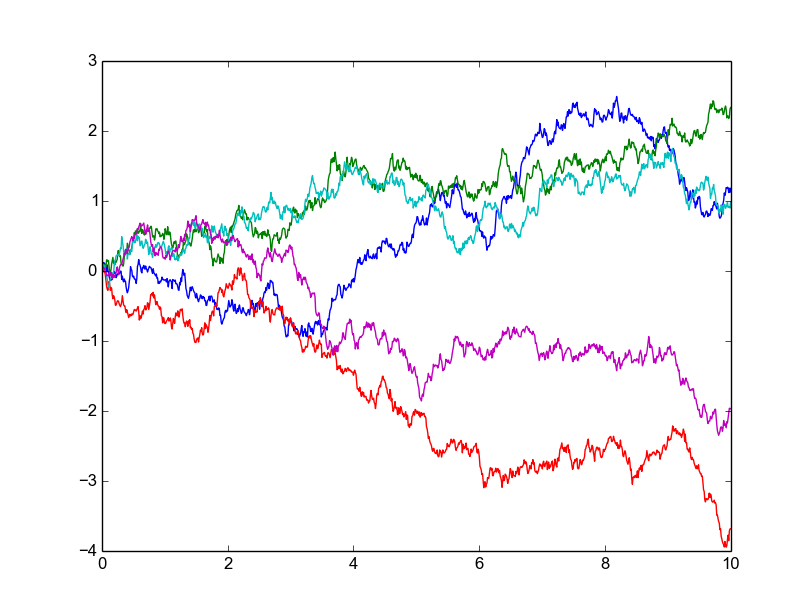}}
 \caption{}
 \end{figure}

\subsection{Bayesian Filtering with \textit{Integrated Ornstein Uhlenbeck process}}\label{sec:FiltIOUP} We analyse the  probabilistic interpretation of ODEs by filtering investigated in \cref{filteringODEs} in the case of a $q$-times integrated Ornstein Uhlenbeck process. \\
To this end, we \textit{a priori} assume that the solution of the ODE and the first $q$  derivatives $(x, x', x'', \dots ,x^{(q)}) : [0,T]\rightarrow \mathbb{R}^{(q+1)\times d}$ follow a $q$-times integrated Ornstein Uhlenbeck process ${X = \allowbreak (X_t)}_{t \in [0,T]}$  $= {({X_t}^{(0)}, \dots , {X_t}^{(q)})}_{t \in [0,T]}$ with initial conditions $X_0 \sim (m_0,P_0)$. Accordingly, the drift matrix $F$ in \cref{eq:matrices} has entries $ \lbrace f_{q,0}, \dots, f_{q,q} \rbrace = \lbrace0, \dots,0, \theta\rbrace  $, $\theta<0$, on the last row. Using this prior model the matrices $A(h)$ and $Q(h)$ can be computed analytically (see appendix) from \cref{eq:A} and \cref{eq:Q} and have the following form:
 \begin{equation}
 A(h)_{i,j} = {e^{hF}}_{i,j} =\begin{cases} \mathbbm{1}_{j\geq i} \frac{h^{j-i}}{(j-i)!}, & \mbox{if }j\neq q \\ \frac{1}{\theta^{q-i}}\left( e^{\theta h} - \sum_{k=0}^{q-i-1}\frac{{(\theta h)}^k}{k!}\right) , & \mbox{if }j=q

\end{cases} 
 \end{equation}
\begin{equation}
\begin{split}
&Q_{i,j}(h) =  \frac{\sigma^2}{\theta^{2q-i-j}} \Bigg\lbrace \frac{e^{2\theta h}-1}{2\theta} - \sum_{k=0}^{q-i-1}\left[  (-1)^{k}\frac{\left(  e^{\theta h} - 1\right)}{\theta}  + \sum_{l=1}^{k} (-1)^{k-l}\frac{\theta^{l-1}e^{\theta h}h^l}{l!}\right]- \\& 
\sum_{k=0}^{q-j-1}{\left[{(-1)}^{k}\frac{\left(  e^{\theta h} - 1\right)}{\theta}  + \sum_{l=1}^k (-1)^{k-l} \frac{\theta^{l-1}e^{\theta h}h^l}{l!}\right]} + \sum_{k_1=0}^{q-i-1}\sum_{k_2=0}^{q-j-1} \frac{{\theta}^{k_1+k_2}}{k_1! k_2!}  \frac{h^{k_1+k_2+1}}{k_1+k_2+1} \Bigg\rbrace
\end{split}
\end{equation}
A probabilistic solver with IOUP prior has local convergence rate for the predicted posterior mean of the same order as the solver with an IWMP prior:
\begin{proposition}
	The order of the prediction in the first step 
	\begin{equation}
		\left \Vert m(h)_0^- - x(h)  \right \Vert \leq Kh^{q+1},
	\end{equation}
	for some constant $K>0$ independent of $h$.

\end{proposition}
\begin{proof}
	\begin{align}
		m^-(h)_0 
		=& 
		\sum_{i=0}^{q+1} A_{0,i}m_i 
		=
		m_0 + \sum_{k=1}^{q-1} \frac{h^k m_k}{k!} + \frac{m^q}{\theta^q} \sum_{k=q}^\infty \frac{{(\theta h)}^k}{k!}
		\\ 
		=&
		m_0 + \sum_{k=1}^{q-1} \frac{h^k m_k}{k!} +m^q \sum_{k=q}^\infty \frac{h^k \theta^{k-q}}{k!}
		\\ 
		=&
		m_0 + hm_1 +\frac{h^2}{2!}m_2 + \dots + \frac{h^{q-1}}{(q-1)!}m_{(q-1)} + \frac{h^q}{q!}m_q + O(h^{q+1})
		\\ 
		=&
		\underbrace{x_0 + hf(0,x_0) + \dots + \frac{h^{q-1}}{(q-1)!}f^{(q-1)}(0,x_0) + \frac{h^q}{q!}f^q(0,x_0)}_{\text{first $q$ components of the Taylor expantion of $x(h)$}} + \mathcal{O}(h^{q+1}) \label{sm7}
	\end{align}
	By writing the Taylor expansion of $x$ and subtracting it from \cref{sm7}, we obtain the desired order for the truncation error.
\end{proof}
As mentioned above, we want to  exploit of the 'mean reverting property' of OUP (and absent in WP), or equivalently of the fact that OUP has bounded variance, for solving ODEs whose trajectories have bounded derivatives. Since we use a $q$-times IOUP prior over the solution, we want that also the integral of the process is maintained the same behaviour, i.e. tends to stay closer to the mean than the WP. This is indeed the case, as we proved that $Var(\int_0^T X_s ds)\leq Var(\int_0^T W_s ds)$ $\forall T>0$ (see appendix).

\section{Experiments} \label{experiments}
In this section we analyse applications of the probabilistic ODE solver with IOUP prior discussed in \cref{sec:FiltIOUP}. In particular we compare  this solver with the one considered in \cref{sec:IWP}, which uses an IWP prior over the solution $x$, by applying the two probabilistic ODE solvers to the same ODEs and examining the plots. The experiments aim to test the intuition that the selection of the prior depends on the structure of the ODE's solution: when the solution looks 'mean-reverting', in the sense that it has bounded derivatives, we expect better results from the IOUP prior, whereas we should choose an IWP prior when solutions to ODEs have unbounded derivatives or look divergent. According to these considerations \cref{table1} shows our prior belief  on which solver (IOUP or IWP) performs better for which of the following ODEs.

\begin{itemize}
\item[1.]{\textit{Linear ODEs}}
\begin{itemize}
\item[\textit{a)}]{Exponential function}
 \begin{flalign*}
     &x'(t) = x \\
     &x(0) = 1 &&    
  \end{flalign*}
\item[\textit{b)}]{Negative exponetial function}
 \begin{flalign*}
     &x'(t) = -x \\
     &x(0) = 1 &&    
  \end{flalign*} 
\end{itemize}
\item[2.]{\textit{Oscillators}}
\begin{itemize}
\item[\textit{a)}]Orbit equations \cite{hull1972comparing} \\
 $y'_0=y_2 \ \ \ \ \ \ \ \ \ \ \ \ \ \ \ \ \ \ \ \ \ \ \ \ \ \ \ \ \  y_0(0)=1-\varepsilon$\\
    $y'_1=y_3 \ \ \ \ \ \ \ \ \ \ \ \ \ \ \ \ \ \ \ \ \ \ \ \ \ \ \ \ \  y_1(0)=0$\\
    $y'_2=-y_0/{(y_0^2+y_1^2)}^{3/2} \  \ \ \ \ \ \ \ \ \  y_2(0)=0$\\
        $y'_3=-y_1/{(y_0^2+y_1^2)}^{3/2} \  \ \ \ \ \ \ \ \ \  y_3(0)=\sqrt{\frac{1+\varepsilon}{1-\varepsilon}}$\vspace{.2cm}\\
$\varepsilon = .1$ (epsilon is the eccentricity of the orbit) \vspace{.2cm}
  \item[\textit{b)}] Van der Pol oscillator \begin{equation*}
  x''(t) -\mu(1-x^2)x'(t) + x = 0 \ \ \ \ x(0)=0
\end{equation*}       
\end{itemize}

\item[3.]{\textit{Moderate systems}} \cite{hull1972comparing} \\
\\
$\begin{pmatrix}
  \ \  x'_0 \  \\
  \ \ x'_1 \    \\
  \vdots  \\
  \vdots \\
   \ \ x'_8 \    \\
  \ \  x'_9 \    
 \end{pmatrix} = \begin{pmatrix}
 -1 &  &  &  &  & &   \\
  1 & -2 &  & & 0 &  \\
      &   2  &-3 &    \\
    &   & \ddots & \ddots &   \\
    
     &  0 & & \ddots & -9 &                                  \\
   &  &  &  &  9 & 0 \\
 \end{pmatrix}$
  $\begin{pmatrix}
  \ \  x_0 \  \\
  \ \  x_1 \    \\
  \vdots  \\
  \vdots \\
   \ \ x_8 \    \\
  \ \  x_9 \    
 \end{pmatrix}$ , \ \ \ \    $x(0) = \begin{pmatrix}
  1 \\
  0 \\
  \vdots \\
  \vdots \\
   0  \\
  0
 \end{pmatrix}$ 
 \\
\end{itemize}

\vspace{.8cm}
For the experiments the IOUP and IWP filters are initialized with  mean  $m_0={(x_0,f(0,x_0), 0, \dots, 0)}^T$  and variance $P_0$ with non-zero entries only in the minor matrix ${P_0}_{ord:(q+1),ord:q+1} = I$ where  \textit{ord} is the order of the ODE and $I \in \mathbb{R}^{(q+1-ord)\times(q+1-ord)}$ is the identity matrix.
\begin{table}[h!]\label{table1}
\centering
\begin{tabular}{|p{6cm}|c|c|} 
\hline
ODE & {\color{teal}IWP-solver} & {\color{red}IOUP-solver}  \\ [0.5ex] 
 \hline
1. a) Exponential function  & $\times$    &\\
1. b) Negative exponential function&  & $\times$ \\
2. Oscillators & & $\times$\\
3. Moderate systems  & & $\times$\\
 \hline 
\end{tabular}
\caption{Expectations of which solver --- IWP or IOUP--- should work better on which ODE, based on the intuition described above }
\label{table:1}
\end{table}

\subsection{Linear ODEs}
The solution of \textit{a)}, the exponential function, has  a divergent trajectory with unbounded derivatives and as we expected  the plot in \cref{fig:linearodes} (a) shows that the IWP-prior solver approximates the solution better than the IOUP-prior.  
On the contrary, the negative exponential function, solution of \textit{b)}, decays to zero. This resembles the behaviour of the mean-reverting samples of the IOU process, which should therefore work as a good prior over $x$. The plot in \cref{fig:linearodes} (b) confirms this intuition: the IOUP-solver approximates better the solution than the IWP-solver. Both methods use a model of step size $h=0.5$.\\
The plot in \cref{fig:linearodes} (b) clearly shows the behaviour of the two solvers: IWP oscillates much more than IOUP around the solution when this converges to zero, because the IWP doesn't tend to 'drift back' to the mean, whereas the IOUP solver remains closer to the solution due to the process drifting property and smaller variance. 

 \begin{figure} \label{fig:linearodes}
 \centering
 \subfigure[Positive exponential, $q=2$, $\theta=-\frac{3}{2}$]
   {\includegraphics[width=6.1cm]{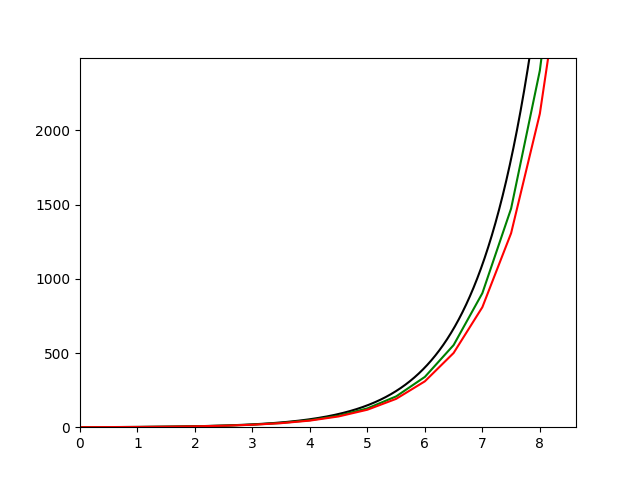}}
 \hspace{5mm}
 \subfigure[Negative exponential, $q=2$, $\theta=-\frac{3}{2}$]
   {\includegraphics[width=6.1cm]{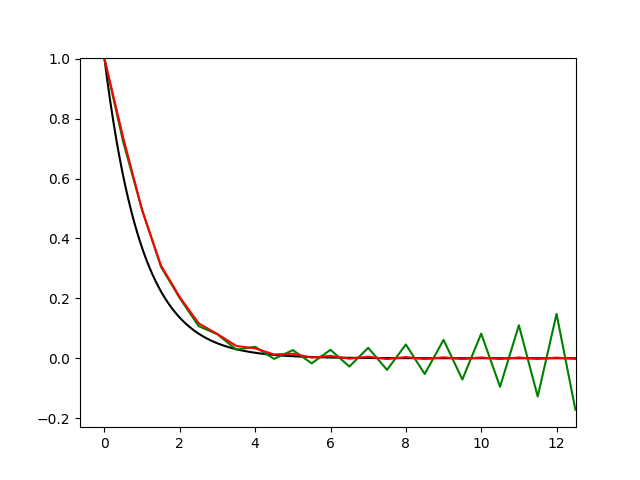}}
 \caption{In red: IOUP-solver, in green: IWP-solver. In black the solution (computed with a Runge Kutta method with very  small step size $h=0.0001$)}
 \end{figure}

\subsection{Oscillators}
We used an IOUP-prior to compute approximate oscillating solutions of ODEs in \textit{Problem 2. a)} and \textit{b)}. The original intuition was that an IOUP-solver would be a good prior for oscillators, because their trajectory resembles the 'mean-reverting' property and looks similar to the samples of an OUP. However, we do not always see better results of the IOUP-approximate solution. \cref{oscillators} depicts two different oscillators: in the plot (a) the IWP solver  distinctly approximates better the solution $y_0$ in \textit{Problem 2 a)}. On the other hand the Van der Pol oscillator (\textit{Problem 2 b)}) illustrated in \cref{oscillators} (b) is better approximated by an IOUP prior.
To motivate these results we argue that the similarity to the OUP samples is a global property of the solution's structure. However, the filter computes the approximate solution along the time axis 'locally', proceeding with small time steps. In particular the IOUP filter predicts with high probability the $q$-th derivative smaller at every time step, whereas the derivative of an oscillator is periodically changing and doesn't decay to zero or converges to a finite number. 

 \begin{figure} \label{oscillators}
 \centering
 \subfigure[ODE solution $y_0$ of \textit{Problem 2 a)}, $q=1$]
   {\includegraphics[width=6.335cm]{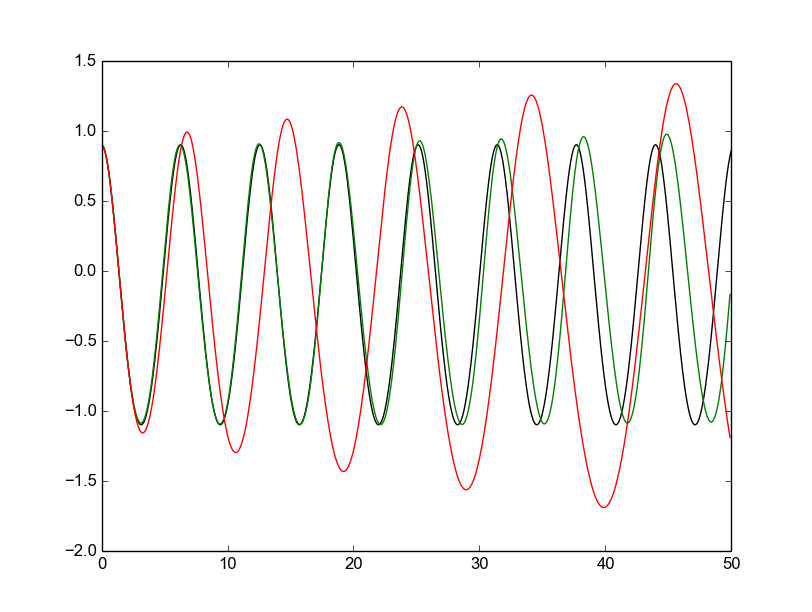}}
 \hspace{0.1mm}
 \subfigure[Van der Pol oscillator. $q=2$, $h=0.05$ ]
   {\includegraphics[width=6.35cm]{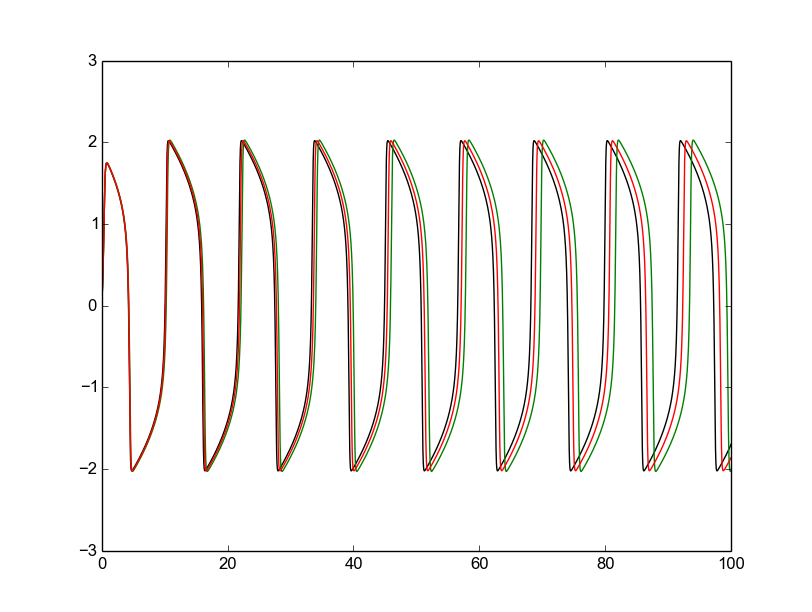}}
 \caption{Oscillators.  IOUP solver in red,  IWP-solver in green.}
 \end{figure}
 
\subsection{Moderate systems}
Problem \textit{3.} is a system of ODEs (of order 1 and dimension 10) which describes a radioactive decay chain. The ten solutions of the ODE problem are shown in \cref{fig:problem3solutions} : the trajectories  decay to zero for $t\rightarrow\infty$ (apart from $x_9(t)$ which converges to one). Again, as discussed for ODE \textit{1} $b)$, the solution functions look similar to the mean reverting samples of the OUP.  In \cref{orbitequations} we plotted two orbits of the ODE problem \textit{3}, with a step size of $h=0.1$ and $q=1$. While the IWP oscillates around the orbit, the IOUP solvers remains near to it. Moreover the more stiff is the orbit, the more we expect the Ornstein-Uhlenbeck to perform well. Indeed, the plots confirm that this intuition might be useful.

\begin{figure}[htbp] \label{fig:problem3solutions}
 \centering
  \includegraphics[width=.50\textwidth]{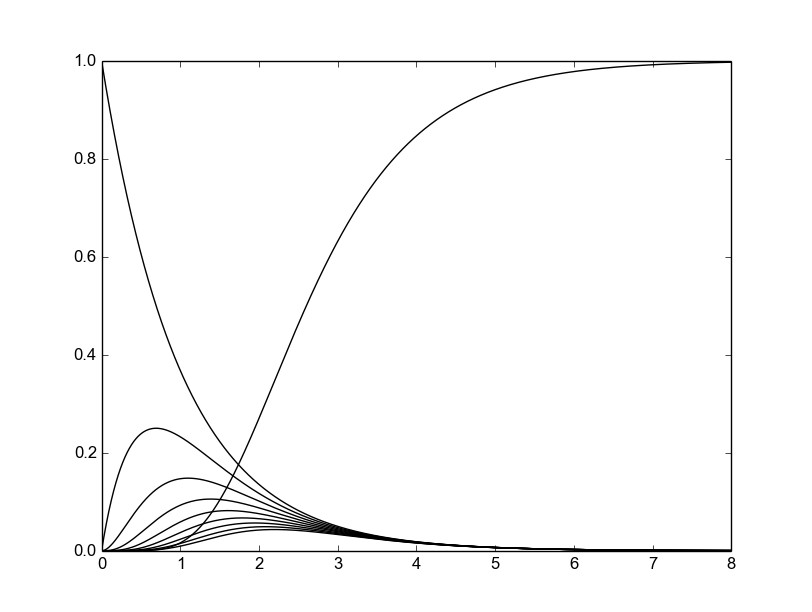}\caption{Solution functions $\lbrace x_0, \dots, x_9\rbrace$ to ODE Problem 3, computed with a Runge Kutta method with step size $h=0.0001$ }
\end{figure}

 \begin{figure} \label{orbitequations}
 \centering
 \subfigure[Orbit with the function $x_8(t)$ on the x-axis and $x_7(t)$ on the y-axis,$t \in (0,10)$ ]
   {\includegraphics[width=6.3cm]{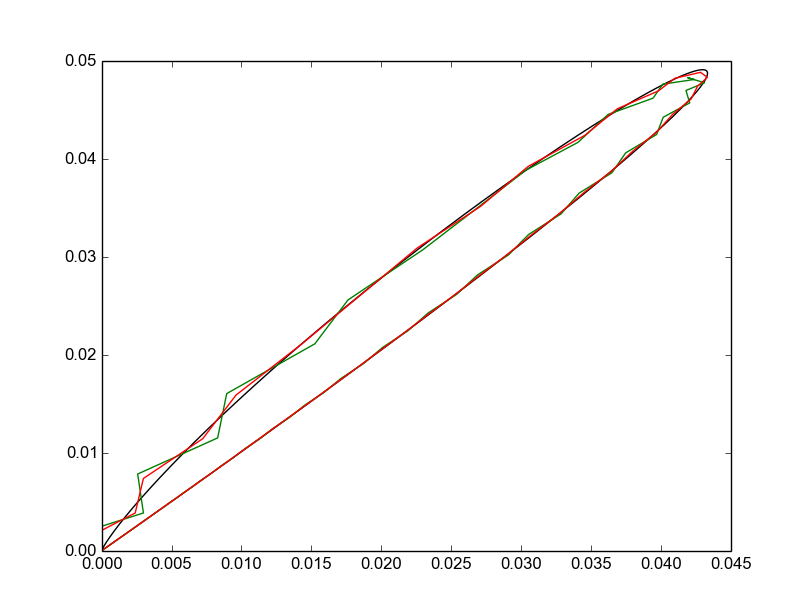}}
 \hspace{0.1mm}
 \subfigure[Orbit with the function $x_0(t)$ on the x-axis and $x_8(t)$ on the y-axis]
   {\includegraphics[width=6.3cm]{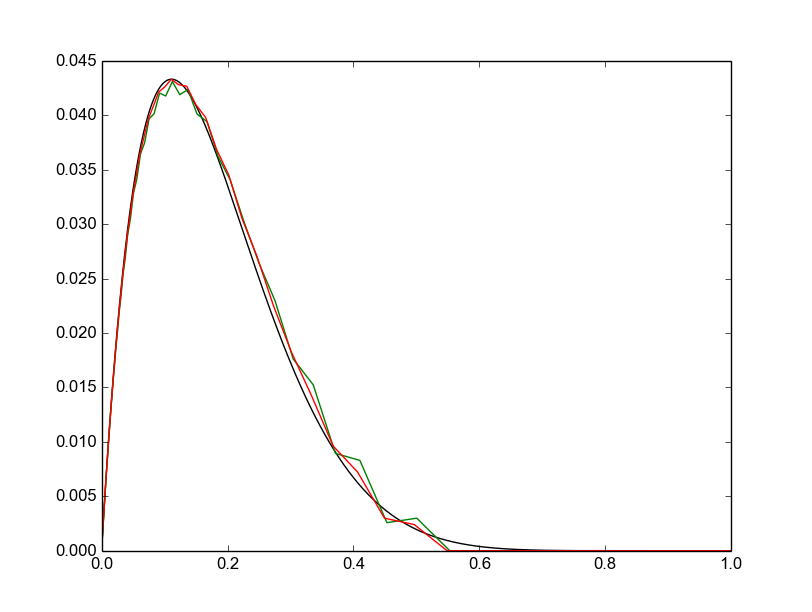}}
 \caption{Orbits of Problem 3. In red: IOUP-solver, in green: IWP solver. In black the solution (computed with a RK-method with very  small step size)}
 \end{figure}

\section{Discussion}
The experiments presented in \cref{experiments} provide good material to discuss the differences between the IOUP and the IWP solver. An IOUP prior appears to approximate better 'ODEs whith  bounded derivatives',
due to the drifting property which keeps the process near to the mean zero with high probability (in our model the $q$-th derivative of the solution). \cref{fig:linearodes} (b) and \cref{orbitequations} suggest this argument is reasonable: it distinctly shows how the approximate solution with an IWP prior oscillates unstably around the analytical solution due to the fact that it is less conservative than the OUP (in the sense that it tends to deviate away from its mean) whereas the IOUP solver remains closer to it.
In this perspective, an IWP works as a better prior for divergent ODEs, like e.g the exponential function (\cref{fig:linearodes} a)).       
Our intuition was not validated for oscillators (\cref{oscillators}). This is probably due to the fact that the trajectory of oscillators--- although it resembles the drifting property--- periodically changes the value of the derivative, whereas the IOUP-solver predicts a smaller derivative at each time step.

\section{Summary}
In the framework for Bayesian filtering we examined theory of probabilistic ODE solvers which return Gaussian process posterior distributions instead of point estimates and provide uncertainty measures over the solutions. We stressed the Bayesian nature of these solvers and discussed the question of prior selection. Hence we proposed a novel solver which employs an IOUP prior over the solution of ODEs, and provided experiments which show a competitive behaviour of this solver with the already well-established IWP-solver.


\bibliographystyle{siamplain}
\bibliography{./bibfile.bib}
\newpage
\section{Appendix} 

\subsection{Matrices $A(h)$ and $Q(h)$ or IOUP}

It is easy to see that
\begin{equation}
 A(h)_{i,j} := {e^{hF}}_{i,j} = \begin{cases}
  \mathbbm{1}_{j\geq i} \frac{h^{j-i}}{(j-i)!}, & \mbox{if }j\neq q \\ \frac{1}{\theta^{q-i}}\left( e^{\theta h} - \sum_{k=0}^{q-i-1}\frac{{(\theta h)}^k}{k!}\right) , & \mbox{if }j=q

\end{cases} 
\end{equation}

Let's calculate $Q(h)$ given by \cite{sarkka2006thesis}
\begin{equation}
Q(h) = \int_0^h exp(F\tau)LQL^Texp{(F\tau)}^T d\tau
\end{equation}

Where \begin{equation*}
F = \begin{pmatrix}
0&1&0&\cdots&0\\
0&0&1&\cdots&0\\
\vdots& & \ddots&\ddots&\vdots \\
0& &\cdots&0&1\\
0&0&\dots&0&\theta\\
\end{pmatrix}   \  \ \  L = \begin{pmatrix}
0\\0\\\vdots\\0\\\sigma\\
\end{pmatrix} 
\end{equation*} 
with $\theta < 0 $. We assume $Q=I$. We have

\begin{equation*}
{LQL^T}_{i,j} = \begin{cases} \sigma^2 & \mbox{if }j=i= q \\  0 & \mbox{otherwise }
\end{cases}
\end{equation*}
so
\begin{equation*}
{exp(F\tau)LQL^T}_{i,j} = \begin{cases} \sigma^2\sum_k^\infty \frac{\theta^k h^{k+q-i}}{(k+q-i)!} = \frac{1}{\theta^{q-i}}\left( e^{\theta h} - \sum_{k=0}^{q-i-1}\frac{{(\theta h)}^k}{k!}\right) & \mbox{if }j= q \\  0 & \mbox{otherwise }
\end{cases}
\end{equation*}
and
\begin{equation*}
{exp(F\tau)LQL^Texp{(F\tau)}^T}_{i,j} =  \sum_{k=0}^{+\infty} \frac{\theta^k \tau^{k+q-i}}{(k+q-i)!} \sum_{k=0}^{+\infty} \frac{\theta^k \tau^{k+q-j}}{(k+q-j)!} 
\end{equation*}

Then
\begin{align*}
&Q_{i,j}(h)= \sigma^2 \int_0^h \Big( \sum_{k=0}^{+\infty} \frac{\theta^k \tau^{k+q-i}}{(k+q-i)!} \Big)\Big( \sum_{k=0}^{+\infty} \frac{\theta^k \tau^{k+q-j}}{(k+q-j)!} \Big) d\tau \\ &=  \sigma^2 \int_0^h \Big( \frac{1}{\theta^{q-i}} \sum_{k=0}^{+\infty} \frac{{(\theta\tau) }^{k+q-i}}{(k+q-i)!} \Big)\Big( \frac{1}{\theta^{q-j}} \sum_{k=0}^{+\infty} \frac{{(\theta\tau) }^{k+q-j}}{(k+q-j)!} \Big) d\tau  \\
&= \sigma^2  \int_0^h \left( \frac{1}{\theta^{q-i}} \left( e^{\theta\tau} - \sum_{k=0}^{q-i-1} \frac{{(\theta\tau)}^k}{k!}\right) \right) \left(\frac{1}{\theta^{q-j}} \left( e^{\theta\tau} - \sum_{k=0}^{q-j-1} \frac{{(\theta\tau)}^k}{k!}\right) \right) \\&= \sigma^2 \int_0^h \left( \frac{e^{2\theta\tau}}{\theta^{2q-i-j}} - \frac{e^{\theta\tau}}{\theta^{2q-i-j}}\sum_{k=0}^{q-i-1} \frac{{(\theta\tau)}^k}{k!} - \frac{e^{\theta\tau}}{\theta^{2q-i-j}}\sum_{k=0}^{q-j-1} \frac{{(\theta\tau)}^k}{k!} + \frac{1}{\theta^{2q-i-j}}\sum_{k=0}^{q-i-1} \frac{{(\theta\tau)}^k}{k!}\sum_{k=0}^{q-j-1} \frac{{(\theta\tau)}^k}{k!}\right)\\& =
 \frac{\sigma^2}{\theta^{2q-i-j}} \left\lbrace   \left( \frac{e^{2\theta h}-1}{2\theta}\right) -  \sum_{k=0}^{q-i-1}\underbrace{\frac{\theta^k}{k!}\int_0^h e^{\theta\tau}\tau^k d\tau}_{\emph{I}_k (h)} -  \sum_{k=0}^{q-j-1}\underbrace{\frac{\theta^k}{k!} \int_0^h e^{\theta\tau}\tau^k d\tau}_{\emph{I}_k (h)} + \sum_{k_1=0}^{q-i-1}\sum_{k_2=0}^{q-j-1}\frac{\theta^{k_1+k_2}}{k_1!k_2!}\int_0^h \tau^{k_1+k_2} d\tau \right\rbrace    
\end{align*}

Writing 
\begin{equation*}
{I}_k (h) = \frac{\theta^{k-1}e^{\theta h}h^k}{k!} - \emph{I}_{k-1} (h)  \ \    \text{with}  \ \ \   {I}_0 (h) = \int_0^h e^{\theta\tau}d\tau = \frac{1}{\theta} \left(  e^{\theta h} - 1\right) 
\end{equation*}
we get \begin{equation}\begin{split}
&Q_{i,j}(h)=  \frac{\sigma^2}{\theta^{2q-i-j}}\Bigg\lbrace \frac{e^{2\theta h}-1}{2\theta} - \sum_{k=0}^{q-i-1}\left[  (-1)^{k}\frac{\left(  e^{\theta h} - 1\right)}{\theta}  + \sum_{l=1}^{k} (-1)^{k-l}\frac{\theta^{l-1}e^{\theta h}h^l}{l!}\right] \\& 
- \sum_{k=0}^{q-j-1}{\left[{(-1)}^{k}\frac{\left(  e^{\theta h} - 1\right)}{\theta}  + \sum_{l=1}^k (-1)^{k-l} \frac{\theta^{l-1}e^{\theta h}h^l}{l!}\right]} + \sum_{k_1=0}^{q-i-1}\sum_{k_2=0}^{q-j-1} \frac{{\theta}^{k_1+k_2}}{k_1! k_2!}  \frac{h^{k_1+k_2+1}}{k_1+k_2+1} \Bigg\rbrace
\end{split}\end{equation}
\subsection{IOUP local convergence rate for the posterior mean}
 If we employ a $q$-times integrated OUP over the solution $x$, where $q \in \mathbb{N}$, the prediction step for the mean with initial value $m= (x_0,f(0,x_0), f'(0,x_0), \dots , f^{(q)}(0,x_0))^T := (m_0, \dots ,m_q)$  
is   $m^-(h) = A(h)m$ where

\begin{equation}
 A(h)_{i,j} = {e^{hF}}_{i,j} = \begin{cases}
  \mathbbm{1}_{j\geq i} \frac{h^{j-i}}{(j-i)!}, & \mbox{if }j\neq q \\ \frac{1}{\theta^{q-i}}\left( e^{\theta h} - \sum_{k=0}^{q-i-1}\frac{{(\theta h)}^k}{k!}\right) , & \mbox{if }j=q \\

\end{cases}  
\end{equation}

\subsection{Variances of IOUP and IWP}
Let $X_t$ denote an Ornstein Uhlenbeck process with  and $W_t$ a standard Wiener process. In this section we will write $-\theta$ with $\theta>0$ for simplicity in the calculations. We want to show :\begin{equation*}
\underbrace{Var(\int_0^T X_s ds )}_{a)} \leq \underbrace{Var(\int_0^T W_s ds )}_{b)}
\end{equation*}
We calculate a) and b):
\begin{itemize}
\item[a)] \begin{align*}
&Var(\int_0^T X_s ds ) = \mathbb{E}\left[ {\left( \int_0^T X_s ds\right) }^2 \right] -{ \left( \mathbb{E}\left[ \int_0^T X_s ds \right]\right)  }^2  = \mathbb{E}\left[ {\left( \int_0^T X_s ds\right) }^2 \right] - \underbrace{{ \left(  \int_0^T \underbrace{\mathbb{E}[ X_s]}_{=0} ds \right)  }^2}_{=0}=\\
& = \mathbb{E}\int_0^T\int_0^T X_s X_t ds dt = \int_0^T\int_0^T \mathbb{E}[X_s X_t] ds dt \stackrel{\text{$\mathbb{E}[X_t]= \mathbb{E}[X_s]=0  $}}{=} \int_0^T\int_0^T Cov(X_s,X_t) ds dt = \\
&=\frac{\sigma^2}{2\theta} \int_0^T\int_0^T(e^{-\theta\mid t-s \mid} - e^{-\theta(t+s)})ds dt = \frac{\sigma^2}{2\theta} \left(  \underbrace{\int_0^T\int_0^Te^{-\theta\mid t-s \mid} ds dt}_{\textit{1)}} -  \underbrace{\int_0^T\int_0^T  e^{-\theta(t+s)}dsdt}_{\textit{2)}}\right) 
\end{align*}
\begin{itemize}
\item [\textit{1)}]\begin{align*}
&\int_0^T\int_0^Te^{-\theta\mid t-s \mid} ds dt = \int_0^T\left( \int_0^s e^{-\theta( s-t )}dt + \int_s^T e^{-\theta( t-s )}dt\right) ds = \int_0^T \left( \frac{1-e^{-\theta s}}{\theta} + \frac{1-e^{-\theta (T-s)}}{\theta} \right)ds =\\
&= \int _0^T \frac{2}{\theta}ds - \int _0^T \frac{e^{-\theta s}}{\theta}ds -\int _0^T \frac{e^{-\theta (T-s)}}{\theta}ds = 
\frac{2T}{\theta} +\frac{e^{-\theta T}-1}{\theta^2} - \frac{1-e^{-\theta T}}{\theta^2} = \frac{2}{\theta}\left(  T +\frac{ e^{-\theta T}-1}{\theta}\right) 
\end{align*}
\item [\textit{2)}]\begin{align*}
&\int_0^T\int_0^T  e^{-\theta(t+s)}dsdt = \int_0^T \left( \frac{1}{\theta} (e^{-\theta s}-e^{-\theta(T+s)})\right) ds = \frac{1}{\theta}\left( \int_0^T e^{-\theta s} - \int_0^T e^{-\theta(T+s)})\right) =\\
&= \frac{1}{\theta}\left(  \frac{1- e^{-\theta T}
}{\theta} - \frac{e^{-\theta T}-e^{-2\theta T}
}{\theta}\right) = \frac{1}{\theta^2}{\left( 1-e^{-\theta T} \right) }^2  
\end{align*}
\end{itemize}

So \begin{align*}
&Var(\int_0^T X_s ds ) = \frac{\sigma^2}{2\theta} \left( \frac{2}{\theta}\left(  T +\frac{ e^{-\theta T}-1}{\theta}\right) - \frac{1}{\theta^2}{\left( 1-e^{-\theta T} \right) }^2  \right) = \frac{\sigma^2}{2\theta^3}\left( 2\theta T +2(e^{- \theta T}-1) - {(1-e^{-\theta T})}^2\right) 
\end{align*}
\item[b)] \begin{align*}
&Var(\int_0^T W_s ds ) = \int_0^T\int_0^T Cov(X_s,X_t) ds dt = \int_0^T\int_0^T min(s,t) ds dt =  \int_0^T \left( \int_0^s t dt + \int_s^T s dt \right)ds  =\\
& =\int_0^T (\frac{s^2}{2} +s(T-s) ) ds = \frac{T^3}{3}
\end{align*}
\end{itemize}
So we can see that 
\begin{equation*}
\frac{\sigma^2}{2\theta^3}\left( 2\theta T +2(e^{- \theta T}-1) - {(1-e^{-\theta T})}^2\right)  \leq \frac{T^3}{3} \ \ \ \ \ \forall T >0.
\end{equation*}

\end{document}